\documentclass[11pt, leqno]{amsart}
\usepackage{mathrsfs}
\usepackage{graphicx}
\usepackage{amsfonts,delarray,amssymb,amsmath,amsthm, a4,a4wide}
\usepackage{latexsym}
\usepackage{epsfig}
\usepackage{color}
\usepackage[T1]{fontenc}
\usepackage{hyperref}

\newtheorem{thm}{Theorem}[section]

\newtheorem{lem}[thm]{Lemma}
\newtheorem{prop}[thm]{Proposition}

\theoremstyle{remark}
\newtheorem{rem}[thm]{Remark}

\theoremstyle{definition}
\newtheorem{defn}[thm]{Definition}

\numberwithin{equation}{section}

\newcommand{\abs}[1]{\left\vert#1\right\vert}

\newcommand{\R}{\mathbb R}

\newcommand{\Cr}{\mathcal C}
\newcommand{\St}{\mathcal S}
\newcommand{\e}{\varepsilon}

\newcommand{\p}{\partial}

\newcommand{\comment}[1]{}

\makeatletter
\@namedef{subjclassname@2020}{%
  \textup{2020} Mathematics Subject Classification}
\makeatother
\begin{document} 

\title[Isolation of scalar Allen--Cahn local minimizers]{Isolation of scalar Allen--Cahn local minimizers
}
\author{Nam Q. Le}
\address{Department of Mathematics, Indiana University, 
Bloomington, IN 47405, USA}
\email{nqle@iu.edu}
\thanks{The author was supported in part by the National Science Foundation under grant DMS-2452320.}



\begin{abstract}
We show that scalar local minimizers of the Allen--Cahn functional with the double-well potential and homogeneous Neumann boundary conditions on a bounded
Lipschitz domain are isolated in the $L^1$-topology.
This affirmatively answers a question raised by  Kohn and Sternberg (Local minimisers and singular perturbations,
\emph{Proc. Roy. Soc. Edinburgh Sect. A} \textbf{111} (1989)).
The proof implements an 
analytic Lyapunov--Schmidt reduction inspired by the theory of monotone
dynamical systems in works of Jiang--Yu and Hirsch--Smith in an elliptic setting.
\end{abstract}
\maketitle

\section{Introduction and statement of the main results}

In this paper, we prove that scalar local minimizers of the Allen--Cahn functional with the double-well potential on a bounded
Lipschitz domain are isolated in the $L^1$-topology. This affirmatively answers a question raised by  Kohn and Sternberg \cite{KS}.

\medskip

Throughout, let $\Omega\subset\R^n$ $(n\geq 2)$ be a bounded domain with
Lipschitz boundary. 

For each $\e>0$, we consider the Allen--Cahn functional  with the double-well potential defined on scalar $L^1(\Omega)$ functions:
\begin{equation} \label{MM} F_{\e}(u)=\left\{
 \begin{alignedat}{1} \int_{\Omega}\Big[\e \abs{D u}^2 +\e^{-1}(1-u^2)^2\Big] \, dx  ~&~ \text{if} ~u\in H^1(\Omega)\cap L^4(\Omega),\\\
 \infty~&~ \text{otherwise}.
 \end{alignedat} 
  \right.
 \end{equation} 
This functional arises in the van der Waals--Cahn--Hilliard gradient theory of phase transitions \cite{AC} and its connection with minimal surfaces via Gamma-convergence. 
 Consider the area functional $E$ defined on $L^1(\Omega)$ by
\begin{equation*}
E(u)=\left\{
 \begin{alignedat}{1}
   \frac{1}{2}\int_{\Omega}|D u| ~&~ \text{if} ~u\in BV (\Omega, \{1, -1\}), \\\
\infty~&~ \text{otherwise}.
 \end{alignedat} 
  \right.
  \end{equation*}
For a function of bounded variation $u\in BV (\Omega, \{1, -1\})$ taking values $\pm 1$, $|D u|$ denotes the total variation of the vector-valued measure $Du$ (see \cite{Simon}).

\medskip

It is a classical result that $F_\e$ Gamma-converges to $(8/3)E$ in the $L^1$-topology; see \cite{KS, Modica,MM,St}. We will just use the consequence of this convergence in the statement of Theorem \ref{isolate1}.

\medskip

Casten--Holland \cite{CH} proved that any nonconstant critical point of $F_\e$ is unstable on a convex domain $\Omega$. In contrast,  
Matano \cite{Matano} established the existence of nonconstant stable critical points of $F_\e$
 in suitable nonconvex domains. 
Based on the Gamma-convergence of $F_\e$ to $(8/3)E$, Kohn--Sternberg \cite{KS} discovered a very interesting connection between isolated local minimizers of the area functional  and the existence of local minimizers of $F_\e$ (which are obviously stable). They proved the following theorem.
\begin{thm}[\cite{KS}, Theorem 2.1]
\label{isolate1}
Let $\Omega$ be a bounded domain in $\R^n$ $(n\geq 2)$ with Lipschitz boundary, and suppose that $u_0$ is an isolated $L^1$-local minimizer of $E$. Then there exists $\e_0>0$ and a family $\{u_{\e}\}_{\e<\e_0}$ such that
\begin{enumerate}
\item $u_\e$ is an $L^1$-local minimizer of $F_\e$;
\item $\lim_{\e\to 0}\|u_\e-u_0\|_{L^1(\Omega)}=0$, and $\lim_{\e\to 0} F_\e(u_\e)= 8E(u_0)/3$. 
\end{enumerate}
\end{thm}
We recall the concept of (isolated) $L^1$-local minimizers. A similar definition applies to $L^\infty$.
\begin{defn}\label{def_local} Let $\e>0$ and $u_0, u\in L^1(\Omega)$.
\begin{enumerate}
\item We call $u_0$ an isolated $L^1$-local minimizer of $E$ if there is $\delta>0$ such that
$$E(u_0)< E(v)\qquad \text{whenever}~0<\|v-u_0\|_{L^1(\Omega)}\leq\delta.$$
\item We call $u$ an $L^1$-local minimizer of $F_\e$ if there is $\delta>0$ such that 
$$F_\e(u)\leq F_{\e}(v)\qquad\text{whenever}~\|u-v\|_{L^1(\Omega)}\leq\delta.$$
  We call $u$ an isolated $L^1$-local minimizer of $F_\e$ if there is $\delta_\ast>0$ such that 
$$F_\e(u)<F_{\e}(v)\qquad\text{whenever}~0<\|u-v\|_{L^1(\Omega)}\leq\delta_\ast.$$
\end{enumerate}
\end{defn}
 Since $C_c^\infty(\Omega)$
is dense in $L^1(\Omega)$ (see \cite[Corollary 4.23]{Brezis}), any $L^1$-local minimizer $u$ of $F_\e$ has finite energy $F_\e(u)$.

\medskip

In \cite[Remark 2.3]{KS}, Kohn and Sternberg asked whether each $u_\e$ in Theorem \ref{isolate1} is isolated. This question has been open ever since. 
In this paper, we give an affirmative answer to it. 

\begin{thm}[Isolation of scalar Allen--Cahn local minimizers]\label{cor:KS}
Each $u_\e$ in Theorem \ref{isolate1} is an isolated $L^1$-local
minimizer of the Allen-Cahn functional $F_\e$.
\end{thm}

Also in \cite{KS}, Kohn and Sternberg
suggested the positivity of the second
variation (see \eqref{Queq}) and its $\e\to 0$ limit
as a possible route to their question. As it turns out, there is a precise relationship in \cite{Le15} between the $\e\to 0$ limit of the second
variation of $F_\e$ and that of the area functional. However, this seems to be far from resolving the isolation question.

\medskip

Here, we address the question of Kohn and Sternberg
at each
fixed $\e>0$ by using the structure of the critical point set. In fact, Theorem \ref{cor:KS} follows from Theorem \ref{thm:main} below.

\medskip

From now on, let $0<\e<1$ be fixed. The upper limit $1$ is chosen only for convenience. The integrand of $F_\e$ in \eqref{MM} contains the double-well potential 
\begin{equation}\label{Weq}
W(u) =(u^2-1)^2/2.\end{equation}
We will state our results and proofs using $W$ defined by \eqref{Weq} but they can carry over to more general potentials; see Remark \ref{Wrem}.

\medskip

Any  $L^1$-local minimizer $u$ of $F_\e$ is a critical point of $F_\e$ with finite energy and belongs to $H^1(\Omega)\cap L^4(\Omega)$. As such, it is a weak solution of
\begin{equation}\label{eq:pde}
 -\e^2\Delta u+2(u^3-u)\equiv  -\e^2\Delta u + W'(u)=0 \quad\text{in }\Omega,
 \qquad \partial_\nu u=0 \quad\text{on }\partial\Omega,
\end{equation}
where $\nu$ is the outward unit normal vector on $\p\Omega$. 

Since $\Omega$ is only Lipschitz,
we should understand \eqref{eq:pde} with the homogeneous Neumann boundary condition as
\begin{equation}\label{eq:weak}
 \e^2\int_\Omega D u\cdot D\psi\, dx
       +\int_\Omega W'(u)\psi\,dx=0
       \qquad\text{for all }\psi\in H^1(\Omega).
\end{equation}
Note that $H^1(\Omega)\subset L^4(\Omega)$ when $2\leq n\leq 4$ by the Sobolev embedding theorem, while $H^1(\Omega)$ need not be contained in $L^4(\Omega)$ when $n\geq 5$.
The criticality of the local minimizer $u$ implies that \eqref{eq:weak} initially holds for all $\psi \in H^1(\Omega)\cap L^4(\Omega)$. However, 
we will prove in Lemma \ref{EL_lem} that  $|u|\leq1$. Therefore, \eqref{eq:weak} holds for all
$\psi\in H^1(\Omega)$, and we can focus on $H^1(\Omega)$, bounded weak solutions of \eqref{eq:pde}. 
\medskip

For a bounded weak solution of \eqref{eq:pde}, we consider 
 the associated linearized operator
\[L_u:=-\e^2\Delta + (6u^2-2)I\equiv -\e^2\Delta  + W''(u) I,\]
 and its second variation quadratic form
\begin{equation}\label{Queq}
 Q_u(\psi)=\int_\Omega\Big[\e^2|D\psi|^2+
                          (6u^2-2)\psi^2\Big]\, dx\equiv \int_\Omega\Big[\e^2|D\psi|^2+
                          W''(u)\psi^2\Big]\, dx.
\end{equation}
We call a weak solution $u\in H^1(\Omega)$ of \eqref{eq:pde} \emph{stable} if it is bounded and satisfies the stability condition
\[Q_u(\psi)\geq 0 \qquad \text{for all }
\psi\in H^1(\Omega).\]

\noindent
Our key result establishes the finiteness and isolation of the set of stable critical points of $F_\e$.

\begin{thm}[Finiteness and isolation of stable critical points of the Allen--Cahn functional]
\label{thm:main}
Let $\Omega\subset\R^n$ $(n\geq 2)$ be a bounded domain with
Lipschitz boundary. Let $0<\e<1$ and $F_\e$ be defined by \eqref{MM}.
Then, the following conclusions hold.
\begin{enumerate}
\item The set of stable weak solutions of \eqref{eq:pde} is finite.
\item Each stable weak solution of \eqref{eq:pde} is isolated among all finite-energy
weak solutions, both in $L^\infty(\Omega)$ and in $L^1(\Omega)$ topologies.
\item Every $L^1$-local minimizer of $F_\e$ is an isolated
$L^1$-local minimizer. In particular, there are only finitely many such
local minimizers.
\end{enumerate}
\end{thm}

We now indicate the ideas in the proof of Theorem \ref{thm:main}.

\medskip
\noindent
{\bf Notation.} Denote by $\Cr$
the set of bounded critical points of $F_\e$ and by $\St\subset \Cr$ the set of stable critical points. We use $I$ for the identity map.
\subsection*{Mechanism of the proof}
Both sets $\Cr$ and $\St$ will be shown to be compact in $L^\infty(\Omega)$. We express bounded critical points of  $F_\e$ as zeros of a map $G$ (see \eqref{Geq1}),  which is a compact perturbation of the identity on $L^{\infty}(\Omega)$. 
At a stable
critical point $u$ of $F_\e$, we will analyze the Fredholm operator $G'(u)$ which is of index zero and has the same kernel as $L_u$;
either the linearized operator $L_u$ is invertible, or its kernel is
one-dimensional and generated by a positive principal eigenfunction.
In the second case, the analytic Lyapunov--Schmidt reduction provides a real analytic
function in one variable. If $u$ is nonisolated, that scalar
function vanishes identically, and we obtain an analytic curve $\gamma$ of
critical points of $F_\e$ through it. 
The key observation is that every point on this curve that is sufficiently close to $u$ is
also stable. The reason is that the tangent vector provides a zero eigenvalue to the linearized operator $L_\gamma$, while the
second eigenvalue stays positive. This produces a curve inside $\St$.
If $\St$ were infinite, its set of accumulation points would be nonempty and compact.
Maximizing the mass functional 
$v\mapsto\int_\Omega v\, dx$
 on that set gives a contradiction,
because the positive tangent to the curve makes the mass increase.
The passage from the isolation in $L^\infty$ to isolated $L^1$-minimality is then simple.

\medskip

The analytic alternative used here is inspired by the theory of monotone
dynamical systems. In particular, Hirsch and Smith
\cite[Proposition 5 and Remark 2, pp.~393--394]{HS} presented a result
attributed to Jiang and Yu \cite[Lemma 3.3]{JY}: under appropriate 
spectral radius hypotheses, a nonisolated fixed point lies on a locally unique
analytic ordered arc of fixed points. Proposition~\ref{prop:arc} en route to the proof of Theorem~\ref{thm:main} is proved
below by an elliptic implementation of this mechanism. (Note that in the erratum to \cite{HS}, the authors
withdraw the assertion that the entire equilibrium set is finite
in their Theorems 1 and 2. Our setting and finiteness results do not rely on these theorems.)
The analytic
implicit function theorem and Lyapunov--Schmidt reduction are standard;
for an introductory treatment, see Deimling
\cite[Sections 15.1 and 16.2]{Deimling} and Kielh\"ofer \cite[Section I.2]{Kh}. For the broader dynamical framework, see Dancer--Hess \cite{DH} and
 Smith \cite{Smith}.

\medskip

\subsection*{AI assistance} The main results of this paper were achieved through a series of inquiries  with OpenAI's GPT-6 Astra. The key strategies using the  analytic Lyapunov--Schmidt reduction inspired by the  works of Jiang--Yu and Hirsch--Smith were obtained by ChatGPT. 
The author checked, reworked, and simplified all arguments and rewrote the proofs of the main results. 
 The author takes full responsibility
for the correctness and content of the paper.

\medskip

The rest of the paper is organized as follows. In Section \ref{spec_sec}, we prove the boundedness of local minimizers, introduce a Neumann resolvent operator,  and analyze spectral properties of the linearized operator associated with a bounded critical point of the Allen--Cahn functional.
The analytic Lyapunov--Schmidt  reduction will be carried out in Section \ref{LS_sec}.
We prove Theorem~\ref{thm:main} in Section \ref{Finite_sec}.

\section{Bounded critical points and spectral analysis of the linearized operator}
\label{spec_sec}
This section is devoted to preparatory materials.  We prove the boundedness of local minimizers, introduce a Neumann resolvent operator,  and analyze spectral properties of the linearized operator associated with a bounded critical point of the Allen-Cahn functional.
\subsection{Local minimizers are bounded stable critical points}

Via standard truncations, we first show that $L^1$-local minimizers of $F_\e$ are bounded stable critical points.

\begin{lem}\label{EL_lem}
Fix $0<\e<1$. Then the following hold.
\begin{enumerate}
\item Every finite-energy weak solution $u$ of \eqref{eq:pde} satisfies $\|u\|_{L^{\infty}(\Omega)}\leq1$ and $u\in C^\infty(\Omega)$.
\item Every $L^1$-local minimizer of $F_\e$ satisfies \eqref{eq:weak}  and is stable.
\end{enumerate}
\end{lem}

\begin{proof}
We prove (i) by first showing that $-1\leq u\leq 1$.  Since $u$ has finite energy, $u\in H^1(\Omega)\cap L^4(\Omega)$. 
By \cite[Lemma 7.6]{GT}, the truncated function $(u-1)_+:=\max\{u-1, 0\}\in H^1(\Omega)\cap L^4(\Omega)$, so we can use it as 
an admissible test function in \eqref{eq:weak}. Thus
\begin{equation}\label{eq:truncate}
 \e^2\int_\Omega|D(u-1)_+|^2\, dx
 +2\int_{\{u>1\}}u(u+1)(u-1)^2\, dx=0.
\end{equation}
This implies that $u\leq1$.
Applying the same argument to $-u$, which satisfies  \eqref{eq:weak},
gives $u\geq-1$. Hence,  $\|u\|_{L^{\infty}(\Omega)}\leq1$. 
An application of elliptic regularity (see \cite[Chapter 8]{GT}) shows that $u\in C^\infty(\Omega)$. Part (i) is proved.

\medskip
We now prove part (ii). Let $u$ be an  $L^1$-local minimizer of $F_\e$, so it has finite energy and $u\in H^1(\Omega)\cap L^4(\Omega)$.
Consider first $\psi\in H^1(\Omega)\cap L^4(\Omega)$. Then, the function
$t\mapsto F_\e(u+t\psi)$ is a polynomial in $t$ and has a local minimum at $t=0$. Therefore,
\[
 \left.\frac{d}{dt}F_\e(u+t\psi)\right|_{t=0} =2\e\int_\Omega D u\cdot D\psi\, dx
 +4\e^{-1}\int_\Omega(u^3-u)\psi\, dx=0,
\]
which is equivalent to \eqref{eq:weak} for all $\psi\in H^1(\Omega)\cap L^4(\Omega)$, and
\begin{equation}\label{eq:secondvariation}
 \left.\frac{d^2}{dt^2}F_\e(u+t\psi)\right|_{t=0}
 =2\e\int_\Omega|D\psi|^2\, dx
  +4\e^{-1}\int_\Omega(3u^2-1)\psi^2\, dx
 =\frac{2}{\e}Q_u(\psi)\geq0.
\end{equation}

By the same truncation argument in part (i), $u$ is bounded, so is $u^3-u$. Bounded $H^1(\Omega)$ truncations are dense in $H^1(\Omega)$,
so \eqref{eq:weak} extends to all $H^1(\Omega)$ test functions. 
The same density argument proves that $Q_u(\psi)\geq0$ for all $\psi \in H^1(\Omega)$. The lemma is proved.
\end{proof}

\subsection{The Neumann resolvent operator $R$ and compactness}
\label{sec_R}

For $g\in L^2(\Omega)$, the Lax--Milgram theorem gives a unique $v\in H^1(\Omega)$
such that
\begin{equation}\label{eq:Rweak}
 \e^2\int_\Omega D v\cdot D\psi\, dx
       +\int_\Omega v\psi\,dx=\int_\Omega g\psi\, dx
       \qquad \text{for all }\psi\in H^1(\Omega).
\end{equation}
Write $v=Rg$ where $R$ is the {\it Neumann resolvent operator}:
\begin{equation}\label{Rdefn}
 R=(-\e^2\Delta_N+I)^{-1},
\end{equation}
where the subscript $N$ indicates the Neumann boundary condition associated with $v$.

\medskip

Letting $\psi=v$ in \eqref{eq:Rweak} shows that $R:L^2(\Omega)\to H^1(\Omega)$ is bounded and
\[\|Rg\|_{H^1(\Omega)}\leq \e^{-2}\|g\|_{L^2(\Omega)}.\] Since $\Omega$ is bounded and Lipschitz, the Rellich--Kondrachov compactness theorem (see \cite[Theorem 7.26]{GT}) then shows that
\begin{equation}
\label{L2cpt}
R:L^2(\Omega)\to L^2(\Omega)\qquad \text{is compact}.\end{equation}
 Moreover, $R$ is self-adjoint on $L^2(\Omega)$: 
\begin{equation}\label{eq:Rselfadjoint}
 (g, Rh)_{L^2(\Omega)}=(Rg, h)_{L^2(\Omega)}\qquad \text{for all }g,h\in L^2(\Omega).
\end{equation}

\medskip

Similarly, consider the rescaled Laplace operator $A$ and its active domain $D(A)$:
\begin{equation}
\label{Adefn}
A=-\e^2\Delta, \quad D(A):=
\{v\in H^1(\Omega): -\e^2\Delta v\in L^2(\Omega),~ \p_\nu v=0  \text{ weakly}\}.
\end{equation}
Explicitly, $v\in D(A)$ if and only if $v\in H^1(\Omega)$ and there
exists $f\in L^2(\Omega)$ such that $Av=f$ in the following weak sense:
\[
 \e^2\int_\Omega D v\cdot D\psi\,dx=\int_\Omega f\psi\, dx \quad \text{for all }\psi\in H^1(\Omega).
\]
In the previous paragraph defining $R$, we have $Av=g-v$ weakly. It can be verified that
\begin{equation}
\label{DAeq}
D(A)= R (L^2(\Omega)).\end{equation}

{\bf Compactness of $R$ in $L^\infty(\Omega)$.}  Similar to \eqref{L2cpt}, 
we now prove the compactness of $R$  in $L^\infty(\Omega)$ using the global boundedness for weak solutions to Schr\"odinger operators
$-\e^2 \Delta -f I$
 with bounded potentials $f$ and right-hand sides and Neumann boundary condition. 
 This compactness allows us to work on the Banach space $L^\infty(\Omega)$ on which the nonlinearity $W'(u)$ is analytic and
bounded critical points of the Allen--Cahn functional $F_\e$ are zeros of the map 
\begin{equation}\label{Geq1} G(v)\equiv v-R(3v-2v^3)\equiv v-R (v-W'(v)),\end{equation} which is a compact perturbation of the identity. 

\begin{lem}[Global boundedness for  weak solutions to Schr\"odinger operators]
\label{lem:Rcompact}
Let $\Omega\subset\R^n$ be a bounded Lipschitz domain. Let
$0<\e<1$,  $f_0,f\in L^\infty(\Omega)$. Assume
$v\in H^1(\Omega)$ satisfies
\begin{equation}\label{f0feq}
 \e^2\int_\Omega D v\cdot D\psi
       +\int_\Omega v\psi
 =\int_\Omega(f_0+fv)\psi
 \qquad\text{for every }\psi\in H^1(\Omega).
\end{equation}
Then $v\in L^\infty(\Omega)$ and there exists a constant $C$, depending only on $\Omega$, $n$, and $\e$, such that  
\begin{equation}\label{interpol1}
 \|v\|_{L^\infty(\Omega)}
 \le C\Big(\|f_0\|^{\frac{n}{n+2}}_{L^\infty(\Omega)}\|v\|^{\frac{2}{n+2}}_{L^2(\Omega)}
                      + \|f\|_{L^\infty(\Omega)}^{\frac{n}{2}}  \|v\|_{L^2(\Omega)}\Big).
\end{equation}
\end{lem}

\begin{proof} The proof is based on De Giorgi's technique; see \cite[Chapter 4]{HL}. 
Set \[F_0=\|f_0\|_{L^\infty(\Omega)},\,\, F=\|f\|_{L^\infty(\Omega)}, \,\, M=\|v\|_{L^2(\Omega)},\,\,  \mu(k)=|\{x\in\Omega:v(x)>k\}|\quad\text{for } k\geq 0.\] 
If $M=0$, then \eqref{interpol1} is obvious.
If $F_0=F=0$, then testing \eqref{f0feq} with $\psi=v$ gives $v=0$.
 Assume now that $M>0$ and $F_0+F>0$.

Since $\Omega$ is bounded and Lipschitz, by \cite[Theorem 7.26]{GT}, there exists a constant $C_\Omega$ depending only on $\Omega$ and $n$, such that the following Sobolev inequality holds:
\begin{equation}\label{eq:Sobolev}
 \|v\|^2_{L^p(\Omega)}\leq C_\Omega\big(\|D v\|^2_{L^2(\Omega)}+
                              \|v\|^2_{L^2(\Omega)}\big)
 \qquad\text{for all } v\in H^1(\Omega),\quad \text{where } p=\frac{2n}{n-1}.
\end{equation}
Set
\[
 \theta=1-2/p=1/n,\qquad
 C_0=C_\Omega/\e^2,\qquad
 C_1=2^{(p-1)/(p-2)}C_0.
\]
Testing \eqref{f0feq} with $(v-k)_+$ and discarding the nonnegative
term $k\int_\Omega (v-k)_+\,dx$ gives
\begin{equation*}
\begin{split}
 \e^2\|D(v-k)_+\|^2_{L^2(\Omega)}+\|(v-k)_+\|^2_{L^2(\Omega)}
 &\leq (F_0+ Fk)\int_{\{v>k\}}(v-k)_+\, dx + F\int_\Omega (v-k)_+^2\, dx\\
 &\leq (F_0+ Fk)\mu(k)^{1-1/p}\|(v-k)_+\|_{L^p(\Omega)} 
 \\
 &\quad+ F\mu(k)^\theta \|(v-k)_+\|^2_{L^p(\Omega)},
 \end{split}
\end{equation*}
where we use the H\"older inequality and note that 
$(v-k)_+$ vanishes outside $\{v>k\}$.

Using \eqref{eq:Sobolev} and $0<\e<1$, we obtain
\begin{equation}\label{eq:beforeabsorption}
 \|(v-k)_+\|^2_{L^p(\Omega)} 
 \le C_0(F_0+Fk)\mu(k)^{1-1/p} \|(v-k)_+\|_{L^p(\Omega)} 
       +C_0F\mu(k)^\theta \|(v-k)_+\|^2_{L^p(\Omega)} .
\end{equation}
Thus, whenever \begin{equation}
\label{absor}
C_0F\mu(k)^\theta\le1/2,\end{equation} we may absorb the last
term in \eqref{eq:beforeabsorption} and then obtain
\begin{equation}\label{eq:truncation}
 \|(v-k)_+\|_{L^p(\Omega)} 
 \le 2C_0(F_0+Fk)\mu(k)^{1-1/p}.
\end{equation}

We now choose a suitable starting level $k_0$ from which the absorption condition holds, and then iterate using \eqref{eq:truncation}. In fact, choose
\begin{equation}\label{eq:kzero}
 k_0:=\max\Big\{
        (4C_1F_0 M^{2\theta})^{1/(1+2\theta)},\,
        (8C_1F)^{1/(2\theta)}M\Big\}>0.
\end{equation}
Then
\begin{equation*}
 4C_1F_0 M^{2\theta}\le k_0^{1+2\theta},\qquad
 8C_1F M^{2\theta}\le k_0^{2\theta}.
\end{equation*}
Together with the Chebyshev inequality,
these relations give
\begin{equation}\label{eq:smallmeasure}
 \mu(k_0)\leq 
 M^2/k_0^2,
 \qquad
 2C_1F_0\mu(k_0)^\theta\le k_0/2,\qquad
 4C_1F\mu(k_0)^\theta\le1/2.
\end{equation}
Observe that the absorption condition \eqref{absor} holds for every $k\ge k_0$ because $C_1\ge C_0$, so
\[
 C_0F\mu(k_0)^\theta
 \le C_1F\mu(k_0)^\theta\leq 1/8.
\]

For $k_0\le k<\ell\le2k_0$, we have
\[F_0+Fk\le F_0+2Fk_0 \qquad\text{and }(v-k)_+\ge\ell-k\quad \text{on }\{v>\ell\}.\]
It follows from \eqref{eq:truncation} that
\begin{equation}\label{eq:iteration}
 \mu(\ell)
 \le\Big(\frac{2C_0(F_0+2Fk_0)}{\ell-k}\Big)^p
       \mu(k)^{p-1}.
\end{equation}

If $\mu(k_0)=0$, then $ \mathop{\rm ess\,sup}_\Omega v\le k_0$.
Otherwise, let
\[
 K:=2C_0(F_0+2Fk_0)>0,\quad
 0<d:=2^{(p-1)/(p-2)}K\mu(k_0)^\theta
   =2C_1(F_0+2Fk_0)\mu(k_0)^\theta\leq k_0,
\]
where we used the last two bounds in \eqref{eq:smallmeasure}.
Consequently, all the levels
\[
 k_j=k_0+d(1-2^{-j}),\qquad j=0,1,2,\ldots,
\]
belong to $[k_0,2k_0]$, where \eqref{eq:iteration} applies.

Since $k_{j+1}-k_j=d2^{-j-1}$,  by induction, \eqref{eq:iteration} gives
\begin{equation}\label{eq:decay}
 \mu(k_j)\le\mu(k_0)2^{-jp/(p-2)}.
\end{equation}
 Since $\{v>k_0+d\}\subset\{v>k_j\}$, letting $j\to\infty$ gives $\mu(k_0+d)=0$.
This implies
\[
  \mathop{\rm ess\,sup}_\Omega v\le k_0+d\le2k_0.
\]
Since $-v$ satisfies \eqref{f0feq} with $f_0$ replaced by $-f_0$ and
$f$ unchanged, the same argument gives $\mathop{\rm ess\,sup}_\Omega(-v)\le2k_0$. Therefore, 
\[\|v\|_{L^\infty(\Omega)}\leq 2k_0.\]
This and \eqref{eq:kzero} together with $\theta=1/n$ establish \eqref{interpol1}.
The lemma is proved.
\end{proof}

\begin{lem}[Compactness of the Neumann resolvent in $L^\infty(\Omega)$]
\label{Rinfcpt}
There exist positive constants $C_\ast=C_\ast(\Omega,\e)$ and $C'_\ast=C'_\ast(\Omega,\e)$
such that
\begin{equation}\label{interpol2}
 \|Rg\|_{L^\infty(\Omega)}\leq
 C_\ast \|g\|^{\frac{n}{n+2}}_{L^\infty(\Omega)}\|Rg\|^{\frac{2}{n+2}}_{L^2(\Omega)}, \,\,  \|Rg\|_{L^\infty(\Omega)}\leq C'_\ast \|g\|_{L^\infty(\Omega)}
 \quad\text{for all } g\in L^\infty(\Omega).
\end{equation}
Consequently, $R:L^\infty(\Omega)\to L^\infty(\Omega)$ is a compact, bounded linear operator.
\end{lem}
\begin{proof} If $g\in L^\infty(\Omega)$, then $v:= Rg\in H^1(\Omega)$ satisfies \eqref{f0feq} with $f_0=g$ and $f\equiv 0$. Thus, the first estimate in \eqref{interpol2} follows from \eqref{interpol1}. Then H\"older's inequality gives the second estimate.
Therefore, $R:L^\infty(\Omega)\to L^\infty(\Omega)$ is a bounded linear operator.

Suppose that $\{g_j\}$ is bounded in $L^\infty(\Omega)$. It is bounded in
$L^2(\Omega)$, so the compactness of $R$ on $L^2(\Omega)$ gives a subsequence, still denoted $\{Rg_j\}$,  which
 converges in $L^2(\Omega)$. Applying the first estimate in \eqref{interpol2} to
$g_j-g_k$, and using $\|R(g_j-g_k)\|_{L^2(\Omega)}\to0$, we find that $\{Rg_j\}$ is Cauchy in
$L^\infty(\Omega)$. This proves the compactness of $R$ on $L^\infty(\Omega)$. The lemma is proved.
\end{proof}

\subsection{Spectral estimates for the linearized operator}

For a bounded critical point $u\in H^1(\Omega)$ of $F_\e$, the associated linearized operator $L_u$ and its corresponding domain $D(L_u)$ for our analysis are given by
\begin{equation}\label{eq:Lu}
 L_u:=-\e^2 \Delta + (6u^2-2)I \equiv A+W''(u) I,\qquad W(u)=(u^2-1)^2/2,\qquad D(L_u)=D(A).
\end{equation}
 Note that the expression
$Q_u(v)$ in \eqref{Queq} is defined for every $v\in H^1(\Omega)$, while
 \[Q_u(v)=(L_u v, v)_{L^2(\Omega)}
\quad\text{for all } v\in D(A).\]

\medskip

As in Section \ref{sec_R}, we can show that $[-\e^2 \Delta + W''(u)I + 3I]^{-1}$ is a  self-adjoint, compact operator on $L^2(\Omega)$. Thus, it has a sequence of eigenvalues $\eta_i(u)$, all positive, with corresponding nontrivial eigenfunctions $\psi_i(u)\in L^2(\Omega)$. The eigenvalue equation implies that
$\psi_i(u)\in D(A)$. Thus, $L_u$ has eigenvalues $\lambda_i(u) =\frac{1}{\eta_i(u)}-3$ and  corresponding nontrivial eigenfunctions $\psi_i(u)\in D(A)$. We list the eigenvalues of $L_u$ with multiplicity as
\[
 \lambda_1(u)\leq\lambda_2(u)\leq\cdots\longrightarrow+\infty.
\]
Choose $\phi_1, \phi_2, \cdots$ to be an $L^2(\Omega)$-orthogonal sequence of eigenfunctions of $L_u$:
\[L_u \phi_j= \lambda_j(u)\phi_j.\]
The first eigenvalue has a variational characterization given by the Rayleigh quotient:
\begin{equation}\label{eq:lambda1}
 \lambda_1(u)=
 \inf_{\psi\in H^1(\Omega)\setminus\{0\}}
 \frac{Q_u(\psi)}{\|\psi\|^2_{L^2(\Omega)}}.
\end{equation}
In particular, the stability of $u$ is equivalent to $\lambda_1(u)\geq0$. 

\medskip

We will need several spectral properties of the linearized operator.
\begin{lem}\label{lem:spectral}
For a bounded critical point $u\in H^1(\Omega)$ of $F_\e$,
every eigenfunction of $L_u$ belongs to $L^\infty(\Omega)$.
The first eigenvalue $\lambda_1(u)$ is simple, and its eigenfunction can be chosen
to be strictly positive in $\Omega$. Moreover, for bounded critical points $u,v\in H^1(\Omega)$ of $F_\e$, we have
\begin{equation}\label{eq:eigencontinuous}
 |\lambda_j(u)-\lambda_j(v)|\leq
 \|W''(u)-W''(v)\|_{L^{\infty}(\Omega)}\leq12\|u-v\|_{L^{\infty}(\Omega)}
 \qquad\text{for all } j\geq 1.
\end{equation}
\end{lem}

\begin{proof}
If $L_u\psi=\lambda\psi$ where $\psi\in D(A)\subset H^1(\Omega)$, then
\begin{equation}\label{eq:eigenweak}
 \e^2\int_\Omega D\psi\cdot D\eta\, dx + \int_\Omega \psi\eta\, dx
 =\int_\Omega(1+ \lambda-W''(u))\psi\eta\, dx
 \qquad\text{for all }\eta\in H^1(\Omega).
\end{equation}
Since $1+\lambda-W''(u)\in L^\infty(\Omega)$,
Lemma \ref{lem:Rcompact} 
gives $\psi\in L^\infty(\Omega)$. 
Interior elliptic regularity then gives $\psi\in C^{\infty}(\Omega)$.

\medskip

The infimum in \eqref{eq:lambda1} is attained by a standard argument using the weak compactness in
$H^1(\Omega)$ and strong compactness in $L^2(\Omega)$ of a minimizing sequence, and the lower semicontinuity of $Q_u(\psi)$. Replacing
a minimizing function by its absolute value preserves its $L^2(\Omega)$ norm
and the value of $Q_u$. Thus, there is a nonnegative first
eigenfunction $\phi$. The interior Harnack inequality  (see \cite[Section 8.8]{GT}) applied to $L_u-\lambda$
implies that any nonzero
nonnegative eigenfunction is strictly positive in the whole domain $\Omega$.

\medskip

For the simplicity of $\lambda_1(u)$, let $\psi$ be any nonzero first eigenfunction. Its
absolute value is also a Rayleigh quotient minimizer and is hence strictly
positive in $\Omega$. Since $\psi$ is continuous there and
$\Omega$ is connected, $\psi$ has one sign. If the first eigenspace
had dimension at least two, one could choose a nonzero first
eigenfunction orthogonal in $L^2(\Omega)$ to $\phi:=|\psi|>0$, contradicting the
one-sign property. This proves the simplicity of $\lambda_1(u)$.

\medskip

Finally,  for bounded critical points $u,v\in H^1(\Omega)$ of $F_\e$, 
the Rayleigh quotients for $Q_u$ and $Q_v$ differ by at most
$\|W''(u)-W''(v)\|_{L^{\infty}(\Omega)}$. The min--max principle for $\lambda_j$
\begin{equation}\label{mimax}
\begin{split}
\lambda_j(u) &= \min_{\substack{E \subset H^1(\Omega) \\ \dim E = j}} \, \max_{\psi \in E \setminus \{0\}} \frac{Q_u(\psi)}{\|\psi\|_{L^2(\Omega)}^2} \\
&= \min_{\substack{\psi \in H^1(\Omega) \setminus \{0\} \\ (\psi, \phi_i)_{L^2(\Omega)} = 0, \, i = 1, \dots, j-1}} \frac{Q_u(\psi)}{\|\psi\|_{L^2(\Omega)}^2}
\end{split}
\end{equation}
 (see \cite[Chapter 6]{Brezis}, \cite[Chapter 6]{Evans} and \cite[Chapter 8]{GT}) gives the first
inequality in \eqref{eq:eigencontinuous}. The second inequality follows from
$|u|,|v|\leq1$ and
\[|W''(u)-W''(v)|=|6u^2-6v^2|\leq12|u-v|.\]
The lemma is proved.
\end{proof}

\medskip

We record some consequences of Lemma \ref{lem:spectral} with normalized first eigenfunction.
At a stable critical point $u$ of $F_\e$, if $0$ is an eigenvalue of $L_u$, then
\begin{equation}\label{eq:simplezero}
 \lambda_1(u)=0<\lambda_2(u),\quad
 \ker L_u=\text{span}\{\phi\}, \quad \phi>0\text{ in }\Omega\quad \text{with } \|\phi\|_{L^2(\Omega)}=1.
\end{equation}
The higher eigenvalue $\lambda_2(u)$ will be used in the proof of Lemma \ref{lem:arcstable} concerning stable arc.

\section{The analytic Lyapunov--Schmidt  reduction and stability}\label{LS_sec}

This section implements an 
explicit 
analytic Lyapunov--Schmidt reduction inspired by the theory of monotone
dynamical systems to study the finiteness and isolation of stable critical points of the Allen--Cahn functional.

\subsection{Critical points as zeros of an analytic perturbation of the identity map}
Let $X=L^\infty(\Omega)$ and define the real analytic map
\begin{equation}\label{eq:G}
 G:X\longrightarrow X,\qquad
 G(v)=v-R(3v-2v^3)\equiv v- R(v-W'(v)),
\end{equation}
where $R$ is defined by \eqref{Rdefn}. We recall also \eqref{Adefn} and \eqref{DAeq}.

A zero $v$ of $G$  satisfies
\[
 v\in D(A),
 \qquad Av+2(v^3-v)\equiv Av + W'(v)=0,
\]
so it is therefore a weak solution of \eqref{eq:pde}. Conversely,
Lemma~\ref{EL_lem} implies that every finite-energy weak solution $v$ of \eqref{eq:pde} is bounded, and thus $G(v)=0$. Hence, we can write
the set of finite-energy critical points of $F_\e$  and its stable subset as
\begin{equation}\label{eq:E}
 \Cr=\{v\in L^\infty(\Omega):G(v)=0\},\qquad
 \St=\{v\in \Cr:\lambda_1(v)\geq0\}.
\end{equation}
These sets are shown to be compact in $L^\infty(\Omega)$.
\begin{lem}\label{CS_cpt}
The sets $\Cr$ and $\St$ are compact in $L^\infty(\Omega)$.
\end{lem}

\begin{proof}
If $v\in \Cr$, then $\|v\|_{L^{\infty}(\Omega)}\leq1$. Thus, for any
sequence $\{v_j\}_{j=1}^\infty\subset \Cr$, the sequence $\{v_j-W'(v_j)\}_{j=1}^\infty$ is 
bounded in $L^\infty(\Omega)$. Lemma~\ref{Rinfcpt} and the identity
$v_j=R(v_j-W'(v_j))$ give a subsequence converging in $L^\infty(\Omega)$.
Its limit belongs to $\Cr$ by the continuity of $G$. This proves
the compactness of $\Cr$. The eigenvalue continuity in
\eqref{eq:eigencontinuous} implies that $\St$ is closed in $\Cr$, so it is
compact as well. The lemma is proved.
\end{proof}

\subsection{A Fredholm operator with index zero associated with a stable critical point}\label{subsec:kernel}
Fix a stable critical point $u\in \St$ of $F_\e$. Let the linear map \[G'(u): L^\infty(\Omega)\to L^\infty(\Omega)\] denote the Fr\'echet derivative of $G$
at $u$. Then
\begin{equation}\label{eq:D}
 G'(u)=I-R [I-W''(u)I].
\end{equation}
We 
reserve the derivative notation such as $D_wH$ for the partial Fr\'echet derivative. The meaning of this notation is that for each $h\in L^\infty(\Omega)$,
\[G'(u) h= h -R(h-W''(u) h).\]
To see the formula for $G'(u)$, recall that the linear operator  $R$ is bounded on $X=L^\infty(\Omega)$. For any Fr\'echet differentiable map $\mathcal F:X\to X$, write
\[
 \mathcal F(u+h)-\mathcal F(u)=\mathcal F'(u)h+r(h),
 \qquad \|r(h)\|_X/\|h\|_X\longrightarrow0.
\]

In particular,  if $ \mathcal N: X\to X$ with $\mathcal N(v)=3v-2v^3\equiv v-W'(v)$, then from
\[
 \mathcal N(u+h)-\mathcal N(u)
 =(3-6u^2)h-6uh^2-2h^3 \equiv (1-W''(u))h-6uh^2-2h^3 .
\]
we find
 $\mathcal N'(u)h=(1-W''(u))h$ and \eqref{eq:D} easily follows from $\|u\|_{L^\infty(\Omega)}\leq 1$ and 
\[
 \begin{split}
 \|G(u+h)-G(u)-(h-R(h-W''(u)h))\|_{L^\infty(\Omega)}&= \|R(6uh^2 + 2h^3)\|_{L^\infty(\Omega)}\\
 &\leq \|R\|
 \Big(6\|h\|^2_{L^\infty(\Omega)} +2\|h\|^3_{L^\infty(\Omega)}\Big).
 \end{split}
\]

\medskip

Since $(1-W''(u))I$ is bounded and $R$ is compact, $R[(1-W''(u))I]:X\to X$ is compact.
The Fredholm alternative for identity minus a compact operator
\cite[Theorem 6.6, p.~160]{Brezis} therefore implies that $G'(u)$ has finite-dimensional kernel $\ker G'(u)$, 
closed range $\text{Ran } G'(u)$ with finite codimension, and
\[
 \dim \ker G'(u)=\dim(X/ \text{Ran } G'(u)).
\]
In other words, 
 $G'(u)$ is Fredholm
of index zero; see
\cite[Section 6.4, p.~168]{Brezis}.

\medskip

We now identify the two kernels of $G'(u)$ and $L_u$, keeping their domains explicit.
\begin{lem} \label{2kers}
With $G'(u)$ defined by \eqref{eq:D}, $L_u$ by \eqref{eq:Lu}, and $D(A)$ by \eqref{Adefn}, we have
\begin{equation}\label{eq:kernel} 
 \ker\!\bigl(G'(u):L^\infty(\Omega)\to L^\infty(\Omega)\bigr)
 =\ker\!\bigl(L_u:D(A)\to L^2(\Omega)\bigr) \subset D(A)\cap L^\infty(\Omega).
\end{equation}
\end{lem}
\begin{proof}
 If $h\in X=L^\infty(\Omega)$ and $G'(u)h=0$, then
\[h=R[(1-W''(u))h].\] Clearly,
$(1-W''(u))h\in L^2(\Omega)$. The identity
$R(L^2(\Omega))=D(A)$ from \eqref{DAeq} 
gives $h\in D(A)$. Applying $A+I$ to the above equation, we obtain
\[
 (A+I)h=(1-W''(u))h,
 \qquad L_u h=(A+W''(u))h=0.
\]
Conversely, suppose $h\in D(A)=D(L_u)$ and $L_uh=0$. Then,
Lemma~\ref{lem:spectral} gives $h\in L^\infty(\Omega)$.
 Moreover, \[(A+I)h=(1-W''(u))h,\] so applying
$R$ gives \[h=R[(1-W''(u))h],\] and hence $G'(u)h=0$ in $X$.  The lemma is proved.
\end{proof}

We record an isolation property of nondegenerate stable critical points.
\begin{lem} [Isolation of nondegenerate stable critical point]
\label{isolem}
Let $u$ be a stable critical point of $F_\e$. If $\lambda_1(u)>0$, then $u$ is isolated, with respect to
the $L^\infty(\Omega)$ topology, in the set of bounded critical points of $F_\e$. 
\end{lem}
\begin{proof}
If $\lambda_1(u)>0$, then $\ker L_u=\{0\}$. Now, Lemma \ref{2kers} and the Fredholm
alternative show that $G'(u)$ is surjective, hence bijective. Its inverse is
bounded by the bounded inverse theorem
\cite[Corollary 2.7, p.~35]{Brezis}. The inverse function
theorem  on Banach spaces \cite[Section 15.1]{Deimling} 
now tells us that $G$ is one-to-one in a
neighborhood of $u$ (with respect to
the $L^\infty(\Omega)$ topology). Since $G(u)=0$, $u$ is the only zero there, so it is isolated as asserted.
\end{proof}

\subsection{Analytic Lyapunov--Schmidt  reduction for nonisolated stable critical point}
We now consider the case of degenerate stable critical points $u$ of $F_\e$ where $\lambda_1(u)=0$.
This applies to the case of nonisolated stable critical points.  We carry out an analytic Lyapunov--Schmidt reduction (see Deimling \cite[Section 16.2]{Deimling} and Kielh\"ofer \cite[Section I.2]{Kh} for the Lyapunov--Schmidt procedure. Its analytic regularity
comes from the analytic implicit function theorem). 
In particular, using the self-adjoint elliptic linearized operator $L_u$, we implement the local analytic mechanism underlying \cite[Proposition 5 and Remark 2]{HS}, attributed there to \cite{JY}. Note that the simplicity and positivity needed in the reduction
come from $L_u$ and stability.

\medskip

Assume $u\in\St$ with $\lambda_1(u)=0$. Let $\phi$ be
the positive normalized eigenfunction in \eqref{eq:simplezero}. We emphasize that $\phi$ depends on $u$. Then \[\ker L_u=\text{span}\{\phi\} \quad\text{and}\quad (A+I)\phi=(1-W''(u))\phi.\]

Define a projection $P: X\longrightarrow X$ in the domain by
\begin{equation}\label{eq:P}
 Ph=(h, \phi)_{L^2(\Omega)}\phi,
 \qquad X_1=\ker P.
\end{equation}
It is bounded on $X$, because $|\Omega|<\infty$ and
$\phi\in L^\infty(\Omega)$. 

For the range, 
introduce the bounded functional on $X$:
\begin{equation}\label{eq:ell}
 \ell(h)=\big((1-W''(u))\phi, h\big)_{L^2(\Omega)}\quad \text{for all }h\in X.
\end{equation}
Since $R((1-W''(u))\phi)=\phi$ and $R$ is self-adjoint on $L^2(\Omega)$,
\begin{equation}
 \label{eq:ellD}
 \begin{split}
 \ell(G'(u)h)
 &=\big((1-W''(u))\phi, h\big)_{L^2(\Omega)}-\big((1-W''(u))\phi, R((1-W''(u))h)\big)_{L^2(\Omega)} \\
 &=\big((1-W''(u))\phi, h\big)_{L^2(\Omega)}-\big(R((1-W''(u))\phi), (1-W''(u))h\big)_{L^2(\Omega)}\\
 &=0.
 \end{split}
\end{equation}
Furthermore, testing $(A+I)\phi=(1-W''(u))\phi$ with $\phi$ using \eqref{eq:eigenweak} with $\lambda=0$ gives
\begin{equation}\label{eq:ellphi}
 \ell(\phi)=\int_\Omega (1-W''(u))\phi^2\, dx
 =\e^2\int_\Omega| D\phi|^2\, dx+\int_\Omega\phi^2\,dx\geq 1.
\end{equation}
Thus $\ell$ is nonzero. Since $\ker G'(u)=\ker L_u=\text{span}\{\phi\}$ (by Lemma \ref{2kers}) and $G'(u)$
has index zero, both $\text{Ran }G'(u)$ and $\ker\ell$ have codimension one.
Equation~\eqref{eq:ellD} therefore implies
\begin{equation}\label{eq:range}
 \text{Ran }G'(u)=\ker\ell=:Y_1.
\end{equation}
Define the codomain projection $\Pi: X\longrightarrow X$ by
\begin{equation}\label{eq:Pi}
 \Pi h=[\ell(h)/\ell(\phi)]\phi.
\end{equation}
Then $I-\Pi$ projects onto $Y_1$. 

The two decompositions in the Lyapunov--Schmidt reduction are
\[
 X=\text{span}\{\phi\}\oplus X_1
 \quad\text{in the domain},\qquad
 X=\text{span}\{\phi\}\oplus Y_1
 \quad\text{in the codomain}.
\]
Observe that
\begin{equation}\label{eq:Diso}
 G'(u)|_{X_1}:X_1\longrightarrow Y_1
 \quad\text{is a bounded isomorphism}.
\end{equation}
Indeed, the injectivity follows from \eqref{eq:kernel}. To see the surjectivity, if
$y=G'(u)x\in Y_1$, then \[x-Px\in X_1\quad\text{and}\quad G'(u)(x-Px)=y.\] In the last equality, we used $\ker G'(u)=\text{span}\{\phi\}$.

The closed subspaces  $X_1=\ker P$ and $Y_1=\ker\ell$ of
the Banach space $X$ are also Banach spaces. The bounded inverse theorem
\cite[Corollary 2.7]{Brezis} therefore applies to this restriction.

\medskip
 We now show that if a stable critical point of $F_\e$ is nonisolated, then there is an analytic curve  of
bounded critical points of $F_\e$ through it. The analytic Lyapunov--Schmidt reduction turns this statement into a real analytic equation.
Recall \eqref{eq:E}, and that $\Cr$
is the set of bounded critical points of $F_\e$, while $\St\subset \Cr$ is the set of stable critical points.  
\begin{prop}[The local analytic alternative]\label{prop:arc}
Suppose $u\in \St$ is nonisolated in $\Cr$ with respect to
$L^\infty(\Omega)$. Then $\lambda_1(u)=0$, and there is an injective
real analytic curve
\[
 \gamma:(-a,a)\longrightarrow \Cr,
 \qquad \gamma(0)=u,\qquad\gamma'(0)=\phi,
\]
where $\phi>0$ is the  first eigenfunction of $L_u$ with normalization $\|\phi\|_{L^2(\Omega)}=1$, with 
\begin{equation}\label{eq:parameter}
 (\gamma(t)-u, \phi)_{L^2(\Omega)}=t,
 \qquad (\gamma'(t), \phi)_{L^2(\Omega)}=1\qquad\text{for all }|t|<a.
\end{equation}
Moreover, every critical point of $F_\e$ sufficiently close to $u$ lies on this curve. 
\end{prop}

\begin{proof}
The nonisolation of $u$ and Lemma \ref{isolem} give
$\lambda_1(u)=0$. Thus, we can use
\eqref{eq:P}--\eqref{eq:Diso}. Define the real analytic map $H: \R\times X_1\longrightarrow Y_1$:
\[
 (t,w)\in\R\times X_1\mapsto H(t,w)=(I-\Pi)G(u+t\phi+w)\in Y_1.
\]
Because $u$ is a critical point of $F_\e$,
$G(u)=0$, and therefore
\[
 H(0,0)=(I-\Pi)G(u)=0.
\]
For $\eta\in X_1$, the chain rule gives
\[
 D_wH(0,0)\eta=(I-\Pi)G'(u)\eta=G'(u)\eta,
\]
since $G'(u)\eta\in Y_1$ and $I-\Pi$ is the identity map on $Y_1$.
Thus, by \eqref{eq:Diso}, \begin{center}$D_wH(0,0)=G'(u)|_{X_1}$ is a bounded isomorphism from $X_1$
onto $Y_1$.\end{center}
By the analytic implicit function theorem
\cite[Theorem 15.3, Section 15.1]{Deimling},
we can find an analytic
function  in a neighborhood of $0\in\R$ \[w: (-a, a)\longrightarrow X_1\]
for which
\begin{equation}\label{eq:rangeequation}
 w(0)=0,\qquad (I-\Pi)G(u+t\phi+w(t))=H(t, w(t))=0.
\end{equation}
Differentiating at $0$ yields
\[
 0=(I-\Pi)G'(u)(\phi+w'(0))=G'(u)w'(0),
\]
because $G'(u)\phi=0$ and $G'(u)$ maps into $Y_1$. Since $w'(0)\in X_1$,
\eqref{eq:Diso} implies $w'(0)=0$.

\medskip
Define the curve $\gamma$ and the real analytic
scalar function $h$ by
\begin{equation}\label{eq:reduced}
 \gamma(t)=u+t\phi+w(t),\qquad
 h(t)=\ell(G(\gamma(t))).
\end{equation}
The second equation in \eqref{eq:rangeequation} and \eqref{eq:Pi} mean that $G(\gamma(t))= \Pi G(\gamma(t))$ lies in
$\text{span}\{\phi\}$. Thus, 
\[G(\gamma(t)) = \alpha(t) \phi,\]
and by applying $\ell$ on both sides, we find $h(t) =\alpha(t) \ell(\phi)$, so $\alpha(t)=h(t)/\ell(\phi)$. It follows that
\begin{equation}\label{eq:scalar}
 G(\gamma(t))=[h(t)/\ell(\phi)]\phi.
\end{equation}
Thus, $\gamma(t)$ is a bounded critical point of $F_\e$ exactly when $h(t)$ vanishes.

\medskip

Now, nonisolation and the identity theorem for a
real analytic function of one variable force $h$ to vanish identically. To see this, 
every $v$ sufficiently near $u$ has a unique representation
\[
 v=u+t\phi+w,
 \qquad t=(v-u, \phi)_{L^2(\Omega)},\qquad w\in X_1.
\]
If $G(v)=0$, then $H(t,w)=(I-\Pi)G(v)=0$, so the local uniqueness in
the implicit function theorem implies $w=w(t)$. Hence, every 
critical point of $F_\e$ that is sufficiently close to $u$ is of the form $\gamma(t)$. If $t=0$, this gives $v=u$.
Moreover, distinct nearby critical points of $F_\e$ have distinct parameters.

\medskip

Since $u$ is nonisolated, we can choose distinct 
$v_j\in \Cr\setminus\{u\}$ tending to $u$. Their corresponding parameters satisfy
$t_j\ne0$, $t_j\to0$, and $h(t_j)=0$. The identity theorem for a
real analytic function of one variable gives $h\equiv0$ near $0$.
Indeed, if its convergent Taylor series at $0$ had a first nonzero
coefficient $a_m$ where $m\geq 1$, then $h(t)=t^m(a_m+o(1))$, which is nonzero for
all sufficiently small $t\ne0$, a contradiction. Hence, all Taylor
coefficients of $h$ vanish, and the analyticity gives the asserted conclusion.
\medskip

After reducing $a$ if necessary, we have $h(t)=0$ for all $|t|<a$, and \eqref{eq:scalar} proves
$\gamma(t)\in \Cr$ for all $|t|<a$. Also $\gamma'(0)=\phi$.
Since $w(t)\in X_1$, the normalization of $\phi$ gives
\eqref{eq:parameter}, which proves the injectivity of $\gamma$ as well. The proof of the proposition is complete.
\end{proof}

Next, we  show that the analytic curve $\gamma$ in Proposition~\ref{prop:arc} stays inside the set of stable critical points $\St$. 
This stability along the curve will be crucial for the variational argument in establishing the finiteness of $\St$ where we need to find another analytic curve in $\St$ that passes through a given nonisolated point on $\gamma$.

\begin{lem}[Stable arc]
\label{lem:arcstable}
After shortening the interval in Proposition~\ref{prop:arc}, every
$\gamma(t)$ is stable. Moreover, $\gamma'(t)$ is a positive
principal eigenfunction of $L_{\gamma(t)}$.
\end{lem}

\begin{proof} 
Each point $\gamma(t)$ on the curve is a weak solution of \eqref{eq:pde} in the sense of \eqref{eq:weak}. 
The analyticity of $\gamma$ and $\gamma'(0)=\phi$ show that $\gamma'(t)$ is bounded. 
We have
\[\gamma(t) =R(\gamma(t)-W'(\gamma(t))).\]
Since \(R:L^2(\Omega)\to H^1(\Omega)\) is bounded, this identity shows that \(\gamma\) is also analytic as an \(H^1(\Omega)\)-valued curve. Consequently, we can differentiate in $t$ the equation
\[
 -\e^2\Delta \gamma(t)+W'(\gamma(t))=0 \quad\text{in }\Omega,
 \qquad \partial_\nu \gamma(t)=0 \quad\text{on }\partial\Omega,
\]
using the weak formulation \eqref{eq:weak} with test functions, and
find
\[
 -\e^2\Delta \gamma'(t)+W''(\gamma(t))\gamma'(t)=0 \quad\text{in }\Omega,
 \qquad \partial_\nu \gamma'(t)=0 \quad\text{on }\partial\Omega,
\]
also in the sense of \eqref{eq:weak}.
This gives
\begin{equation}\label{Lieq}
 L_{\gamma(t)}\gamma'(t)=0.
\end{equation}

At $t=0$ we have $\lambda_1(u)=0<\lambda_2(u)$. By
\eqref{eq:eigencontinuous}, for all sufficiently small $|t|$,
\begin{equation}\label{eq:secondgap}
 \lambda_2(\gamma(t))>\lambda_2(u)/2>0.
\end{equation}
Since \eqref{eq:parameter} implies $\gamma'(t)\ne0$, equation~\eqref{Lieq} shows that each $L_{\gamma(t)}$ has a zero eigenvalue at each small $t$. If
$\lambda_1(\gamma(t))<0$, then the zero eigenvalue would have index
at least $2$, contradicting \eqref{eq:secondgap}. Hence
$
 \lambda_1(\gamma(t))=0$, which shows that
 $ \gamma(t)\in \St$ is stable.
 
The zero eigenspace of each $L_{\gamma(t)}$ is one-dimensional. Its nonzero elements
have one sign.  By \eqref{eq:parameter},
\[(\gamma'(t), \phi)_{L^2(\Omega)}=1.\] By \eqref{eq:simplezero}, $\phi>0$ in $\Omega$, so $\gamma'(t)>0$ in $\Omega$. The lemma is proved.
\end{proof}

\section{Finiteness and isolation}\label{Finite_sec}
In this section, we prove Theorem~\ref{thm:main} which implies Theorem \ref{cor:KS}.

Observe the rate of change at $0$ of the mass functional on each $\gamma(t)$ equals $\int_\Omega\phi\, dx$ and is thus positive. We will use this fact to show  that the set of stable critical points of $F_\e$ is finite.

\begin{prop}\label{Finite_prop} 
Let $\Omega\subset\R^n$ $(n\geq 2)$ be a bounded domain with
Lipschitz boundary. Let $0<\e<1$ and $F_\e$ be defined by \eqref{MM}.
Let $\St$ be the set of stable critical points of $F_\e$. Then
\begin{enumerate}
\item The set $\St$ 
 is finite.
 \item Every member of
$\St$ is isolated in the set $\Cr$ of critical points of $F_\e$  in each of the $L^\infty(\Omega)$ and $L^1(\Omega)$ topologies.
\end{enumerate}
\end{prop}

\begin{proof}
Suppose first that the compact set $\St$ is infinite.
Consider its set of accumulation points
\[
 \St'=\{u\in \St:\ u\text{ is a limit of distinct points of }\St\}.
 \]
 By the sequential
compactness implied by Lemma \ref{CS_cpt}, $\St'$
is nonempty. Since  $\St$  is compact, its set of accumulation points $\St'$ is also compact.

\medskip

Clearly, the following total mass functional is continuous on $L^\infty(\Omega)$:
\begin{equation*}
 M(v)=\int_\Omega v\,dx.
\end{equation*}
Choose $u_*\in \St'$ maximizing $M$ on $\St'$. Then $u_*$ is
stable and nonisolated in $\Cr$. (Were $u_*$ to be chosen as a maximizer of $M$ over $\St$, it would not be guaranteed to be nonisolated.) 
Let $\phi_*$
be the positive first eigenfunction of $L_{u_*}$, normalized by
$\|\phi_*\|_{L^2(\Omega)}=1$. 
Then, the local analytic alternative in Proposition~\ref{prop:arc} and the stability arc
Lemma~\ref{lem:arcstable} provide an injective analytic curve
\[
 \gamma:(-a,a)\longrightarrow \St,
 \qquad\gamma(0)=u_*,\qquad\gamma'(0)=\phi_\ast>0.
\]
Every point of this curve belongs to $\St'$. Indeed, for any
$t_0\in(-a,a)$, choose distinct $t_j\to t_0$ with $t_j\ne t_0$.
The points $\gamma(t_j)$ are distinct members of $\St$ converging
in $L^\infty(\Omega)$ to $\gamma(t_0)$.

The maximality of $u_*$ 
implies that 
the differentiable function $t\mapsto M(\gamma(t))$ has a local
maximum at the interior point $0$, so its derivative there equals zero. However, this contradicts 
\begin{equation*}
 \left.\frac{d}{dt}M(\gamma(t))\right|_{t=0}
 =\int_\Omega\phi_\ast\, dx>0.
\end{equation*}
Thus $\St$ is finite, and part (i) is proved.

\medskip

We now prove part (ii). For the $L^\infty(\Omega)$ topology, 
if any $u\in \St$ were nonisolated in $\Cr$ in the $L^\infty(\Omega)$ topology, the same
local analytic alternative and stability arc lemma would produce infinitely many
points of $\St$. This contradicts the finiteness just proved.

For the $L^1(\Omega)$ topology, we observe that 
if the sequence
$\{v_j\}\subset \Cr$ converges in $L^1(\Omega)$ to  $u\in \Cr$, then by Lemma \ref{CS_cpt}, every subsequence has a further subsequence converging in
$L^\infty(\Omega)$. Its limit must be $u$,
by the uniqueness of the $L^1(\Omega)$ limit. From this, we find  
that the original sequence $\{v_j\}$ converges in $L^\infty(\Omega)$ as well. From the isolation in the $L^\infty(\Omega)$ topology, it follows that each $u\in \St$ is isolated in $\Cr$
in the $L^1(\Omega)$ topology.  The proposition is proved.
\end{proof}

We are now ready to complete the proof of Theorem~\ref{thm:main}.
\begin{proof}[Proof of Theorem~\ref{thm:main}] Parts (i) and (ii) of Theorem~\ref{thm:main} were proved in Proposition \ref{Finite_prop}. It remains to prove part (iii). 

\medskip

Let $u$ be an 
$L^1$-local minimizer of $F_\e$, so it has finite energy. By Lemma~\ref{EL_lem}, $u\in \St$.
Choose $\delta>0$ as in Definition~\ref{def_local}. Since $\St$ is
finite by Proposition \ref{Finite_prop}, the number
\[
 m=\min\big\{\|w-u\|_{L^1(\Omega)}:w\in \St,\ w\ne u\big\}
\]
is positive, with the convention $m=+\infty$ if this set is empty.
Choose \[\delta_\ast=\tfrac12\min\{\delta,m\}>0.\]

We show that $u$ is an isolated
$L^1$-local minimizer.
Suppose by contradiction that there exists a function $v\in L^1(\Omega)$ satisfying
\begin{equation}\label{equalFe}
 0<\|v-u\|_{L^1(\Omega)}\leq \delta_\ast,
 \qquad F_\e(v)=F_\e(u).
\end{equation}
 Choose
\[0<\delta_1<\delta-\|v-u\|_{L^1(\Omega)}.\] If $\|z-v\|_{L^1(\Omega)}<\delta_1$,
the triangle inequality gives $\|z-u\|_{L^1(\Omega)}<\delta$, and hence
\[
 F_\e(z)\geq F_\e(u)=F_\e(v).
\]
Thus, $v$ is itself an $L^1$-local minimizer of $F_\e$. Lemma~\ref{EL_lem} then
implies $v\in \St$. Therefore,  \[\|v-u\|_{L^1(\Omega)}\geq m>\delta_\ast,\] contradicting 
\eqref{equalFe}.  Hence, $u$ must be an isolated
$L^1$-local minimizer. This proves part (iii) and completes the proof of Theorem~\ref{thm:main}.
\end{proof}
\begin{rem} 
\label{Wrem}
Several remarks are in order.
\begin{enumerate}
\item Our proof of the finiteness and isolation of the set of scalar stable critical points of the Allen--Cahn functional with the double-well potential $W(u)=(u^2-1)^2/2$ and integrand $\e |Du|^2 + \e^{-1} (1-u^2)^2$ can be naturally extended 
to many other analytic potentials provided the set of critical points is uniformly bounded. For example, this is the case of potentials $ (u^{2m}-K)^{2k}$ where $K>0$ and $m$ and $k$ are positive integers. 

\item The identity theorem for real analytic functions of one variable that was used in the proof of  Proposition \ref{prop:arc} (in particular, from an accumulating sequence of zeros to identical vanishing)  has no counterpart for arbitrary smooth functions or for functions of several variables. Thus, one might wonder what happens to the conclusions of Theorem \ref{thm:main} when 
the double-well potential $W(u)=(u^2-1)^2/2$ is replaced by an arbitrary smooth function excluding flat intervals on its graph.  

One can recognize that both the finiteness and isolation fail for the case of $C^{\infty}$ nonnegative and quadratic growth potential
\[\tilde W(u)=u^4 e^{-1/u^2} \sin^2 (1/u)\quad\text{if } u\neq 0,\quad \tilde W(0)=0,\]
with no flat intervals on its graph.
In this case, the constants \[u_0\equiv 0, \quad u_k=\frac{1}{k\pi}~(k=1, 2,\cdots)\]
	are all global minimizers and stable critical points of the Allen-Cahn functional with potential $\tilde W$ and  \(u_k\to u_0\) uniformly.
\item Our arguments crucially use the simplicity and nonnegativity of the principal eigenvalue together with the one-sign property of the principal  eigenfunction of the linearized operator at a stable critical point of the scalar Allen--Cahn functional. 
 They do not automatically carry over to the vector-valued or volume-constrained Allen--Cahn stable critical points mentioned in \cite[Section 3]{KS}, and taken up in other works; see  \cite{ASt, MStZ, St2, StZ} for examples most close to our Allen--Cahn functional. 
\end{enumerate}
\end{rem}
\medskip

\noindent
{\bf Acknowledgements.} The author would like to thank Peter Sternberg for his interest and helpful comments.


\begin{thebibliography}{9999}
\bibitem{ASt}
Adimurthi, A. and Sternberg, P.
Local minimizers in 3d of vector Allen-Cahn with a quadruple junction.
\emph{Calc. Var. Partial Differential Equations} \textbf{65} (2026), no. 3, Paper No. 84, 25 pp.

\bibitem{AC}Allen, S. M. and Cahn, J. W. A microscopic theory for antiphase boundary motion and its application to antiphase domain coarsening.
{\it Acta Metall.} {\bf 27} (1979), 1085--1095.
\bibitem{Brezis}
Brezis, H.
\emph{Functional analysis, Sobolev spaces and partial differential equations}.
Universitext, Springer, New York, 2011.


\bibitem{CH}
Casten, R. G. and Holland, C. J.
Instability results for reaction diffusion equations with Neumann boundary
conditions.
\emph{J. Differential Equations} \textbf{27} (1978), no. 2, 266--273.

\bibitem{DH}
Dancer, E. N. and Hess, P.
Stability of fixed points for order-preserving discrete-time dynamical systems.
\emph{J. Reine Angew. Math.} \textbf{419} (1991), 125--139.

\bibitem{Deimling}
Deimling, K.
\emph{Nonlinear functional analysis}.
Springer-Verlag, Berlin, 1985.

\bibitem{Evans} Evans, L. C.
\emph{Partial differential equations}.
Second edition.
Grad. Stud. Math., 19
American Mathematical Society, Providence, RI, 2010

\bibitem{GT}
Gilbarg, D. and Trudinger, N. S.
\emph{Elliptic partial differential equations of second order}. Reprint of the 1998 edition.
Classics Math.
Springer-Verlag, Berlin, 2001.


\bibitem{HL} Han, Q. and Lin, F. H. \emph{Elliptic partial differential equations}, 2nd ed., Courant Lecture Notes in Mathematics, vol. 1, Courant Institute of Mathematical Sciences/AMS, New York, 2011. 



\bibitem{HS}
Hirsch, M. W. and Smith, H. L. 
Asymptotically stable equilibria for monotone semiflows.
\emph{Discrete Contin. Dyn. Syst.} \textbf{14} (2006), no.~3, 385--398.
See also the erratum available at
\url{https://math.la.asu.edu/~halsmith/erataAIMS_journals.pdf}







\bibitem{JY}
Jiang, J.-F. and Yu, S.-X.
Stable cycles for attractors of strongly monotone discrete-time dynamical
systems.
\emph{J. Math. Anal. Appl.} \textbf{202} (1996), 349--362.

\bibitem{Kh}Kielh\"ofer, H.
\emph{Bifurcation theory.}
An introduction with applications to partial differential equations. Second edition.
Appl. Math. Sci., 156
Springer, New York, 2012.

\bibitem{KS}
Kohn, R. V. and Sternberg, P.
Local minimisers and singular perturbations,
\emph{Proc. Roy. Soc. Edinburgh Sect. A} \textbf{111} (1989),
no.~1--2, 69--84.

\bibitem{Le15}
Le, N. Q.
On the second inner variations of Allen--Cahn type energies and applications
to local minimizers.
\emph{J. Math. Pures Appl.} (9) \textbf{103} (2015), no.~6, 1317--1345.


\bibitem{Matano}
Matano, H.
Asymptotic behavior and stability of solutions of semilinear diffusion equations.
\emph{Publ. Res. Inst. Math. Sci.} \textbf{15} (1979), no.~2, 401--454.


\bibitem{Modica}
Modica, L.
The gradient theory of phase transitions and the minimal interface criterion.
\emph{Arch. Rational Mech. Anal.} \textbf{98} (1987), no.~2, 123--142.


\bibitem{MM}
Modica, L. and Mortola, S.
Un esempio di $\Gamma^{-}$-convergenza.
\emph{Boll. Un. Mat. Ital. B} (5) \textbf{14} (1977), no. 1, 285--299.

\bibitem{MStZ} Montero, J. A.,  Sternberg, P. and Ziemer, W. P.
Local minimizers with vortices in the Ginzburg-Landau system in three dimensions.
\emph{Comm. Pure Appl. Math.} {\bf 57} (2004), no. 1, 99--125.

\bibitem{Simon} Simon, L. \emph{Lectures on geometric measure theory}. Proceedings of the Centre for Mathematical Analysis, Australian National 
University, 3. Australian National University, Centre for Mathematical Analysis, Canberra, 1983.

\bibitem{Smith}
Smith, H. L. 
\emph{Monotone dynamical systems.} An introduction to the theory of competitive and cooperative systems. Mathematical Surveys and Monographs, vol.~41,
American Mathematical Society, Providence, RI, 1995.

\bibitem{St}
Sternberg, P. 
The effect of a singular perturbation on nonconvex variational problems.
\emph{Arch. Rational Mech. Anal.} \textbf{101} (1988), no.~3, 209--260.

\bibitem{St2} Sternberg, P. Vector-valued local minimizers of nonconvex variational problems. Current directions in nonlinear partial differential equations (Provo, UT, 1987). \emph{Rocky Mountain J. Math.} \textbf{ 21} (1991), pp. 799--807. 



\bibitem{StZ} Sternberg, P. and Zeimer, W.
Local minimisers of a three-phase partition problem with triple junctions.
\emph{Proc. Roy. Soc. Edinburgh Sect. A} \textbf{124} (1994), no. 6, 1059--1073.


\end{thebibliography}
\end{document}